\documentclass[12pt]{amsart}

\usepackage[T1]{fontenc}
\usepackage[utf8]{inputenc}
\usepackage{lmodern}
\usepackage{amsmath,amssymb,amsthm,mathtools,amscd}
\usepackage{tikz-cd}
\usepackage{enumitem}
\usepackage{etoolbox}
\usepackage{hyperref}
\usepackage{microtype}
\usepackage[margin=1.1in]{geometry}
\hypersetup{
  hidelinks,
  pdftitle={Spinor-Cone Classes and the Chow Characteristic Image of Spin(10)},
  pdfauthor={Sanghoon Baek},
  pdfsubject={Chow characteristic images of spin and special Clifford groups},
  pdfkeywords={Chow rings, classifying spaces, spin groups, Clifford groups,
    Steenrod operations}
}
\makeatletter
\patchcmd{\abstract}{\normalfont\Small}{\normalfont\normalsize}{}{}
\makeatother

\newtheorem{theorem}{Theorem}[section]
\newtheorem{proposition}[theorem]{Proposition}
\newtheorem{lemma}[theorem]{Lemma}
\newtheorem{corollary}[theorem]{Corollary}

\newcommand{\CH}{\operatorname{CH}}
\newcommand{\Ch}{\operatorname{Ch}}
\newcommand{\Spin}{\operatorname{Spin}}
\newcommand{\OGr}{\operatorname{OGr}}
\newcommand{\Gm}{\mathbb G_m}
\newcommand{\Spec}{\operatorname{Spec}}
\newcommand{\im}{\operatorname{Im}}
\newcommand{\codim}{\operatorname{codim}}
\newcommand{\St}{\operatorname{St}}
\newcommand{\F}{\mathbb F_2}
\newcommand{\Z}{\mathbb Z}

\newcommand{\res}{\operatorname{res}}
\newcommand{\PP}{\mathbb P}
\newcommand{\Even}{\mathcal E}

\title[The Chow Characteristic Image of \(\Spin(10)\)]{The Chow Characteristic Image of \(\Spin(10)\) via the Affine Cone over the Spinor Variety}
\author{Sanghoon Baek}
\address{Department of Mathematical Sciences, KAIST, 291 Daehak-ro,
Yuseong-gu, Daejeon 34141, Republic of Korea}
\email{sanghoonbaek@kaist.ac.kr}
\urladdr{https://mathsci.kaist.ac.kr/~sbaek/}
\date{}
\subjclass[2020]{14C25, 20G15}
\keywords{Chow rings, classifying spaces, spin groups, Clifford groups, Steenrod operations}

\begin{document}

\begin{abstract}
Let \(G=\Spin(10)\) be the split spin group over a field of characteristic
different from \(2\), and let \(T\subset G\) be a split maximal torus. We
determine the image of the integral Chow restriction map
\(\CH(BG)\to \CH(BT)^W\), equivalently \(\CH(BG)\) modulo torsion.
The main new geometric ingredient in the proof is a construction
of the class \(c_2c_3c_5\), where
the \(c_i\) are the elementary Chern classes after restriction to \(T\).  This class is obtained from the proper equivariant push-forward associated
with the affine cone over the spinor variety in its half-spin embedding for
the special Clifford group \(\Gamma^+(10)\).

Modulo two, the image is the subring generated, over the smallest
Steenrod-stable subring of \(\F[c_2,c_3,c_4,c_5]\) containing
\(c_2^2,c_3^2,c_4^2,c_5\) and \(c_2c_3c_5\), by the torus restriction of the
top Chern class of a half-spin representation. The integral characteristic
image is the full inverse image of this mod-two subring under reduction
modulo \(2\).
\end{abstract}

\maketitle

\section{Introduction}

Let \(F\) be a field of characteristic different from \(2\).  Let \(G\) be a
split reductive algebraic group over \(F\), with split maximal torus \(T\) and
Weyl group \(W\).  We use Chow groups of classifying spaces in the sense of
Totaro~\cite{TotaroChow}.  The restriction 
\(\CH(BG)\to\CH(BT)\) has image in \(\CH(BT)^W\); we denote the resulting map by
\[
  \Phi_G:\CH(BG)\longrightarrow\CH(BT)^W .
\]  
We write \(\Ch(-)=\CH(-)/2\CH(-)\). Throughout, \(\F\) denotes the field with two elements, and
\(\varphi_G:\Ch(BG)\to\Ch(BT)^W\) denotes the mod-two restriction homomorphism,
corestricted to the Weyl-invariant subring.  By Totaro's torsion-index theorem, the kernel and cokernel of \(\Phi_G\)
are killed by the torsion index of \(G\)
\cite[Theorem~1.3(1)]{TotaroSpin}. Since \(\CH(BT)\) is torsion-free, the kernel of \(\Phi_G\) is precisely the
torsion ideal of \(\CH(BG)\). Thus determining \(\CH(BG)\) modulo torsion is equivalent to determining the
restriction image.  

For spin and Clifford groups this image is much subtler
than the full Weyl-invariant ring, even after reduction modulo two; see, for
example,
\cite{BaekSpinorCones,KarpenkoEnvelopes,KarpenkoSpinModTorsion,KarpenkoSpecialClifford}.
Complete Chow-ring computations for spin classifying spaces are scarce and
become complicated quickly. Guillot's computation of \(\CH(B\!\Spin(7))_{(2)}\)
already illustrates this complexity; his discussion of \(\Spin(8)\) explains
the additional difficulty of controlling multiplicative relations
\cite[Remark~9.4]{GuillotSpin7}. The \(\Spin(8)\) case over the complex numbers
was later described by Molina Rojas \cite{MolinaRojasSpin8}. This scarcity and
complexity make the characteristic image a natural first target for higher spin
groups, since it identifies with \(\CH(BG)\) modulo torsion.

Karpenko determined the mod-two Chow characteristic image for \(G=\Spin(n)\)
with \(n=7,8,9,11,12,13\): in these cases the image is the subring of squares
in \(\Ch(BT)^W\) \cite[Theorem~1.1]{KarpenkoSpinModTorsion}. The remaining value \(n=10\) is exceptional. The torus restriction of the top Chern class of a
half-spin representation and the Euler class \(c_5\) already give non-square classes;
the latter belongs to the image in this rank by
\cite[Lemma~A.1]{KarpenkoSpinModTorsion}. Steenrod operations applied to these classes produce further non-square classes. After these standard classes are accounted for,
one remaining direction is represented by \(c_2c_3c_5\).  The new geometric
construction in the present paper realizes this class as an actual mod-two
Chow characteristic class.  Once this class has been constructed, the remaining
argument is internal and structural: Steenrod stability leaves only three possible residue directions in the coefficient algebra,
and two low-degree Steenrod operations eliminate all of them.

Our first result is a general geometric construction for special Clifford
groups. Fix an integer \(n\ge4\), let \(\Gamma=\Gamma^+(2n)\), and let
\(T_\Gamma\subset\Gamma\) be the standard split maximal torus, with Weyl group
\(W\) and \(\Ch(BT_\Gamma)=\F[z,x_1,\ldots,x_n]\). Let \(\res\) denote the restriction map
\begin{equation}\label{eq:torus-restriction}
  \res:\Ch(B\Gamma)\longrightarrow\Ch(BT_\Gamma)
\end{equation}
induced by \(T_\Gamma\hookrightarrow\Gamma\).  Let \(S^+\) be a half-spin representation of \(\Gamma\), let \(\ell\subset S^+\) be the highest-weight line, and let \(P\subset\Gamma\)
be its stabilizer.  Set
\[
  \Even_n=\{I\subset\{1,\ldots,n\}: |I|\text{ is even and } |I|\ge4\}.
\]
The affine cone over the spinor variety \(\Gamma/P\subset\PP(S^+)\), in its
half-spin embedding, gives a proper \(\Gamma\)-equivariant push-forward
\(J:\Ch_P^r(\Spec F)\to \Ch^{r+\delta_n}(B\Gamma)\) for every \(r\ge0\), where
\(\delta_n=2^{n-1}-1-\binom n2\). Set
\begin{equation}\label{eq:spinor-cone-En}
  E_n=\prod_{I\in\Even_n}\left(\sum_{i\in I}x_i\right).
\end{equation}

\begin{theorem}\label{thm:intro-spinor-cone}
For every \(n\ge4\), with \(\Gamma\), \(T_\Gamma\), \(J\), and \(\res\) as above
and \(E_n\) as in \eqref{eq:spinor-cone-En}, if
\(u\in\Ch(BT_\Gamma)^W\) and \(\widetilde u\in\Ch_P(\Spec F)\) is any class
whose restriction to \(T_\Gamma\) is \(u\), then
\[
  \res J(\widetilde u)=uE_n.
\]
Consequently,
\[
  E_n\cdot\Ch(BT_\Gamma)^W\subset
  \operatorname{Im}\bigl(\Ch(B\Gamma)\to\Ch(BT_\Gamma)^W\bigr).
\]
\end{theorem}

We apply this theorem to the exceptional rank-five case.  Let
\(\Gamma=\Gamma^+(10)\), and write \(c_i=e_i(x_1,\ldots,x_5)\) for the elementary
symmetric torus classes on \(T_\Gamma\).  In rank five, \eqref{eq:spinor-cone-En} gives
\[
  E_5=c_5+c_1c_4+c_1^2c_3+c_1^3c_2.
\]
The class \(c_2c_3E_5\) restricts along the inclusion
\(\Spin(10)\hookrightarrow\Gamma\) to \(c_2c_3c_5\), because the induced map on
maximal tori sends \(c_1\) to zero and fixes \(c_2,c_3,c_4,c_5\). This gives the following immediate consequence for \(\Spin(10)\).

\begin{corollary}\label{cor:rank-five-class}
Let \(T\subset\Spin(10)\) be the standard split maximal torus.  Then
\[
  c_2c_3c_5\in
  \operatorname{Im}\bigl(\Ch(B\!\Spin(10))\longrightarrow\Ch(BT)^W\bigr).
\]
\end{corollary}

We use the mod-two Steenrod operations on Chow groups; see, for example,
\cite{Brosnan}.  We write \(\St=\sum_{i\ge0}\St^i\) for the total
operation, so that \(\St^i\) denotes its \(i\)-th component.  On classifying spaces these operations are applied through Totaro approximations
and commute with restriction to tori; see
\cite{TotaroChow} and \cite[Section~4]{KarpenkoSpinModTorsion}. A graded subring
\(R\) of a mod-two Chow ring is called \emph{Steenrod-stable} if
\(\St^i(R)\subset R\) for every \(i\ge0\).  The smallest Steenrod-stable
graded subring containing a collection of homogeneous elements is obtained by
adjoining all their iterated Steenrod operations and all products.

The main application of the construction is the following complete image
theorem.  Let \(t\in\Ch(BT)\) denote the torus restriction of the top Chern
class of a half-spin representation.  Let \(M\subset\F[c_2,c_3,c_4,c_5]\) be
the smallest graded Steenrod-stable subring containing
\[
  c_2^2,\qquad c_3^2,\qquad c_4^2,\qquad c_5,\qquad c_2c_3c_5.
\]

\begin{theorem}\label{thm:intro-full-image}
Let \(G=\Spin(10)\), and let \(T\subset G\) be the standard split maximal torus. Then
\[
  \operatorname{Im}\varphi_G=M[t].
\]
If \(\rho:\CH(BT)^W\to\Ch(BT)^W\) is reduction modulo two, then
\[
  \operatorname{Im}\Phi_G=\rho^{-1}(M[t]).
\]
Equivalently, via restriction to \(T\), \(\CH(B\!\Spin(10))/\mathrm{tors}\)
is identified with the subring \(\rho^{-1}(M[t])\subset\CH(BT)^W\).
\end{theorem}

The proof of the containment \(\operatorname{Im}\varphi_G\subset M[t]\) is
structural, rather than a degree-by-degree computation. It begins with the rank-five ambient bound of
Lemma~\ref{lem:rank-five-ambient}, which places \(\operatorname{Im}\varphi_G\)
inside \(U[t]\), where
\begin{equation}\label{eq:UR-subrings}
  U=\F[c_2,c_4,c_5,c_3^2,c_3c_5],
  \qquad
  R=\F[c_2^2,c_3^2,c_4^2].
\end{equation}
Using Lemma~\ref{lem:M-mod-c5}, Lemma~\ref{lem:U-normal-form} shows that every
class in \(U/M\) has a unique representative
\(A c_2+B c_4+C c_2c_4\), with \(A,B,C\in R\). For a class \(x\in U\) with residue
\(A c_2+B c_4+C c_2c_4\), the condition \(\St^1(x)\in U\) forces
\(A=C=0\), and the additional condition \(\St^3(x)\in U\) forces \(B=0\). Since \(\operatorname{Im}\varphi_G\) is Steenrod-stable, this eliminates all
possible residues outside \(M\). Finally, Lemma~\ref{lem:halfspin-low-chern} handles the half-spin variable
\(t\) coefficientwise: the restricted Chern classes
\(s_i=c_i(S^+)|_T\) vanish for \(1\le i<8\), so
\(\St^i(t^r)=0\) for \(1\le i<8\) and all \(r\ge0\).

Section~2 constructs the spinor-cone push-forward and computes its torus
restriction by a normal-bundle top Chern class calculation. Section~3 uses the Demazure push-forward formula to obtain Weyl-invariant
multiples of the class \(E_n\), completing the proof of
Theorem~\ref{thm:intro-spinor-cone}. We then prove
Corollary~\ref{cor:rank-five-class} by specializing this construction to
\(\Gamma^+(10)\).  Section~4 proves the mod-two image theorem and the integral statement modulo torsion, beginning with the rank-five ambient bound and the Steenrod formulas
used in the calculation.

\section{The affine cone over the spinor variety and its push-forward}
\label{sec:cone}

Throughout this section, fix \(n\ge4\). We omit \(n\) from the notation whenever
no confusion is possible; the rank-dependent symbols \(E_n\) and \(\delta_n\)
retain their subscripts.

Let \(\Gamma=\Gamma^+(2n)\) be the split special Clifford group. We use the
standard central-product presentation
\[
  \Gamma\simeq (\Gm\times \Spin(2n))/\mu_2,
\]
with semisimple part \(\Spin(2n)\) and quotient torus \(\Gm\); see
\cite[\S23.A]{KMRT} and
\cite[Section~1, (1.1)]{KarpenkoSpecialClifford}. Let \(T_\Gamma\subset \Gamma\) be the
standard split maximal torus, and choose standard character coordinates
\(z,x_1,\ldots,x_n\) on \(T_\Gamma\).  Thus
\[
  \CH(BT_\Gamma)=\Z[z,x_1,\ldots,x_n],
  \qquad
  \Ch(BT_\Gamma)=\F[z,x_1,\ldots,x_n].
\]
Let \(c_i=e_i(x_1,\ldots,x_n)\). Unless explicitly stated otherwise, the
symbols \(c_i\) denote these elementary symmetric torus classes; in particular,
the Euler class is \(c_n=x_1\cdots x_n\).

Let \(W\) be the Weyl group of \(\Gamma\). It is of type \(D_n\), generated by
permutations of the \(x_i\) and even sign changes; in the special Clifford
torus, the sign change in the coordinates \(x_i,x_j\) sends
\(z\) to \(z-x_i-x_j\).

We use the special-Clifford torus coordinates and Weyl action of
\cite[Section~1]{KarpenkoSpecialClifford}. For completeness, we include a
short self-contained proof of the mod-two invariant-ring formula. Let \(S^+\) be the half-spin representation of \(\Gamma\) whose highest weight
is \(z\).

\begin{lemma}\label{lem:special-clifford-invariants}
In the above coordinates,
\[
  \Ch(BT_\Gamma)^W=\F[t,c_1,c_2,\ldots,c_n],
\]
where
\[
  t=\prod_{\substack{I\subset\{1,\ldots,n\},\, |I|\ \mathrm{even}}}
  \left(z+\sum_{i\in I}x_i\right).
\]
The class \(t\) is the mod-two torus restriction of
\(c_{2^{n-1}}(S^+)\). In particular, \(\deg t=2^{n-1}\).
\end{lemma}

\begin{proof}
Let \(A=\F[x_1,\ldots,x_n]\), and let
\(H=\{\sum_{i\in I}x_i:\ |I|\equiv0\pmod 2\}\). Modulo two, the subgroup of
even sign changes fixes the \(x_i\) and acts on \(z\) by the translations
\(z\mapsto z+h\), \(h\in H\). We first compute \(A[z]^H\).

Choose an \(\F\)-basis \(h_1,\ldots,h_d\) of \(H\), put
\(H_j=\langle h_1,\ldots,h_j\rangle\), and set
\(t_j=\prod_{h\in H_j}(z+h)\). Thus \(t_0=z\). We prove by induction that
\(A[z]^{H_j}=A[t_j]\). The case \(j=0\) is clear.

Assume \(A[z]^{H_{j-1}}=A[t_{j-1}]\). The translate of \(t_{j-1}\) by \(h_j\)
is \(H_{j-1}\)-invariant, hence lies in \(A[t_{j-1}]\). It has the same
\(z\)-degree and leading coefficient as \(t_{j-1}\), so it is
\(t_{j-1}+a_j\), where evaluating at \(z=0\) gives
\(a_j=\prod_{h\in H_{j-1}}(h_j+h)\). Since \(h_j\notin H_{j-1}\), all factors
\(h_j+h\) are nonzero linear forms; hence \(a_j\) is a nonzerodivisor in \(A\).

The induced involution on \(A[t_{j-1}]\) sends \(t_{j-1}\) to
\(t_{j-1}+a_j\). Since \(a_j\) is a nonzerodivisor, the invariant subring of this involution is
\(A[t_{j-1}(t_{j-1}+a_j)]\): division by the invariant element
\(t_{j-1}(t_{j-1}+a_j)\) and induction on the \(t_{j-1}\)-degree reduce this to
checking a remainder \(r+s t_{j-1}\), and invariance forces \(sa_j=0\), hence
\(s=0\). Thus
\[
  A[z]^{H_j}=A[t_{j-1}(t_{j-1}+a_j)]=A[t_j].
\]
Induction gives \(A[z]^H=A[t]\), where \(t=\prod_{h\in H}(z+h)\).

Finally, \(\mathfrak S_n\) permutes the \(x_i\), preserves \(H\), and fixes the
orbit product \(t\). Since \(A^{\mathfrak S_n}=\F[c_1,\ldots,c_n]\), we get
\[
  \Ch(BT_\Gamma)^W=(A[z]^H)^{\mathfrak S_n}
  =A^{\mathfrak S_n}[t]=\F[t,c_1,\ldots,c_n].
\]
The \(T_\Gamma\)-weights of \(S^+\) are
\(z-\sum_{i\in I}x_i\), with \(|I|\) even. Their first Chern classes reduce
modulo two to \(z+\sum_{i\in I}x_i\). Therefore \(t\) is the mod-two torus
restriction of \(c_{2^{n-1}}(S^+)\). Since there are \(2^{n-1}\) even subsets
of \(\{1,\ldots,n\}\), one has \(\deg t=2^{n-1}\).
\end{proof}

With the above choice
of torus coordinates, its \(T_\Gamma\)-weights are
\[
  \lambda_I=z-\sum_{i\in I}x_i,
  \qquad
  I\subset\{1,\ldots,n\},\quad |I|\equiv0\pmod 2.
\]
Let \(\ell\subset S^+\) be the weight line of weight \(z\), corresponding to
\(I=\varnothing\), and let \(P\subset\Gamma\) be its stabilizer. Then
\(X=\Gamma/P\) is the closed highest-weight orbit in \(\PP(S^+)\). Since the
connected center of \(\Gamma\) acts by scalars on \(S^+\), the \(\Gamma\)-orbit
of \([\ell]\) coincides with the corresponding \(\Spin(2n)\)-orbit.
Let \(V\) be the split quadratic space of dimension \(2n\) over \(F\), so that
\(\Gamma=\Gamma^+(V)\). Then \(X\) identifies with one connected component
\(\OGr^+(n,V)\) of the maximal orthogonal Grassmannian
\cite[Section~1]{KarpenkoSpecialClifford}.
The inclusion \(X\subset\PP(S^+)\) is the half-spin embedding.

Let \(C\subset S^+\) be the affine cone over \(X\subset\PP(S^+)\). The total space of the tautological line bundle on \(X\) is
\(\widetilde C=\Gamma\times_P\ell\). Identifying it with the incidence variety
\(\{(L,v)\in X\times S^+:v\in L\}\) shows that the natural map
\(\widetilde C\to C\subset S^+\) is projective, hence proper; it is
\(\Gamma\)-equivariant and an isomorphism over \(C\setminus\{0\}\). Since
\(\dim S^+=2^{n-1}\) and \(\dim X=\binom n2\), the affine cone has dimension
\(\binom n2+1\), and hence
\[
  \delta_n:=\codim_{S^+}C=2^{n-1}-1-\binom n2.
\]

Let \(j:\widetilde C\to S^+\) be the composite. Since \(j\) is proper and
\(\dim S^+-\dim\widetilde C=\delta_n\), its integral equivariant push-forward
is defined by \cite[Section~2.3, Proposition~3]{EdidinGraham}. Reducing modulo
two, we obtain $J:\Ch_\Gamma^r(\widetilde C)
  \longrightarrow
  \Ch_\Gamma^{r+\delta_n}(S^+)$. Equivariant homotopy invariance and the standard change-of-groups
isomorphism, applied through the mixed-quotient construction of
\cite[Section~2.2]{EdidinGraham}, give
\[
  \Ch_\Gamma(\widetilde C)
  \simeq \Ch_\Gamma(\Gamma/P)
  \simeq \Ch_P(\Spec F)=\Ch(BP),
  \qquad
  \Ch_\Gamma(S^+)\simeq \Ch(B\Gamma).
\]
Hence \(J\) may be viewed as
\begin{equation}\label{eq:spinor-cone-pushforward}
  J:\Ch_P^r(\Spec F)\longrightarrow \Ch^{r+\delta_n}(B\Gamma).
\end{equation}
Under these identifications, \(1=[\Spec F]_P\in\Ch_P^0(\Spec F)\) corresponds to
the \(\Gamma\)-equivariant fundamental class \([\widetilde C]_\Gamma\). Since \(j\) is birational onto the integral cone \(C\), one has
\(J(1)=j_*[\widetilde C]_\Gamma=[C]_\Gamma\). All computations in this section are carried out modulo two. The passage to the
integral image for \(\Spin(10)\) will be made in Section~4.

We now compute the torus restriction of \(J(1)\), with \(\res\) as in
\eqref{eq:torus-restriction}.

\begin{proposition}\label{prop:normalization}
In \(\Ch(BT_\Gamma)\), one has \(\res J(1)=E_n\), with \(E_n\) as in \eqref{eq:spinor-cone-En}.
\end{proposition}

\begin{proof}
Set
\(\Theta=\res J(1)\in\Ch^{\delta_n}(BT_\Gamma)^W\), and use the same symbol
for the corresponding class in \(\Ch_{T_\Gamma}^{\delta_n}(S^+)\).
Put \(C^\circ=C\setminus\{0\}\). Its restriction to
\(S^+\setminus\{0\}\) is the \(T_\Gamma\)-equivariant fundamental class of
\(C^\circ\). Since
\(C^\circ\simeq\widetilde C\setminus X\) is smooth, the closed immersion
\(C^\circ\hookrightarrow S^+\setminus\{0\}\) is regular of codimension
\(\delta_n\). Let $N=N_{C^\circ/(S^+\setminus\{0\})}|_{\ell^\times}$. The equivariant self-intersection formula
\cite[Section~2, property~(4)]{EdidinGrahamLocalization}, followed by
restriction to \(\ell^\times\), gives
\[
  \Theta|_{\ell^\times}=c_{\delta_n}^{T_\Gamma}(N).
\]

Since \(\ell\) has \(T_\Gamma\)-weight \(z\), the localization sequence for the zero
section of the equivariant line \(\ell\) gives
\[
  \Ch_{T_\Gamma}(\ell^\times)=\Ch_{T_\Gamma}(\Spec F)/(z).
\]
Thus restriction to \(\ell^\times\) means imposing \(z=0\).

The \(T_\Gamma\)-representation \(S^+\) splits as the direct sum of its
one-dimensional weight spaces.  Let \(U\subset V\) be the maximal isotropic
subspace corresponding to \([\ell]\).  The tangent space to the ordinary
Grassmannian at \([U]\) is \(\operatorname{Hom}(U,V/U)\).  The quadratic form
identifies \(V/U\) with \(U^\vee\), and differentiating the isotropy condition
identifies the tangent space to \(X=\OGr^+(n,V)\) with the skew-symmetric
homomorphisms.  Hence $T_{[\ell]}X\simeq\bigwedge\nolimits^2 U^\vee$. With the chosen torus
coordinates, its weights are
\[
  \lambda_{\{i,j\}}-\lambda_\varnothing=-(x_i+x_j),
  \qquad 1\le i<j\le n.
\]
Their first Chern classes in the mod-two equivariant Chow ring are therefore
\(x_i+x_j\).

The tangent direction along \(\ell^\times\) is the weight direction indexed by
\(I=\varnothing\). Along \(\ell^\times\), the tangent bundle of \(C^\circ\) is
the direct sum of the ambient weight-line subbundles indexed by
\(I=\varnothing\) and by the two-element subsets \(I\). Consequently, \(N\) is
the direct sum of the remaining weight-line bundles, indexed by the even
subsets \(I\) with \(|I|\ge4\). For such \(I\), the corresponding character is
\[
  \lambda_I=z-\sum_{i\in I}x_i.
\]
After restriction to \(\ell^\times\) and reduction modulo two, its first Chern
class is \(\sum_{i\in I}x_i\). Hence
\[
  \Theta|_{\ell^\times}=E_n
  \quad\text{in }\Ch_{T_\Gamma}(\ell^\times),
\]
or equivalently \(\Theta\equiv E_n\pmod z\).

It remains to remove the congruence. By Lemma~\ref{lem:special-clifford-invariants},
\(\Ch(BT_\Gamma)^W=\F[t,c_1,\ldots,c_n]\), with
\(\deg t=2^{n-1}\).  Since
\[
  \deg \Theta=\delta_n=2^{n-1}-1-\binom n2<2^{n-1},
\]
the class \(\Theta\) has no \(t\)-term and therefore lies in
\(\F[c_1,\ldots,c_n]\). The polynomial \(E_n\) is symmetric in \(x_1,\ldots,x_n\) and is independent
of \(z\), so \(E_n\in\F[c_1,\ldots,c_n]\). The intersection of the
ideal \((z)\subset \F[z,x_1,\ldots,x_n]\) with
\(\F[c_1,\ldots,c_n]\subset\F[x_1,\ldots,x_n]\) is zero. Therefore the
congruence modulo \(z\) implies \(\Theta=E_n\).
\end{proof}

\section{The spinor-cone projection formula and the \texorpdfstring{$\Spin(10)$}{Spin(10)} class}

Throughout this section, \(n\ge 4\) remains fixed, and we use the notation
introduced in Section~\ref{sec:cone}.  Choose a Borel subgroup \(B\) with
\(T_\Gamma\subset B\subset P\), so that \(\ell\) is the highest-weight line of \(S^+\).
Thus \(P\) is the standard maximal parabolic attached to the chosen spin node.

For a simple root \(\alpha\) with respect to \(B\), with corresponding
reflection \(s_\alpha\), let
\[
  \partial_\alpha(f)=(f-s_\alpha(f))/\alpha.
\]
If \(w=s_{\alpha_1}\cdots s_{\alpha_r}\) is a reduced decomposition, set
\(\partial_w=\partial_{\alpha_1}\circ\cdots\circ\partial_{\alpha_r}\).
By Demazure~\cite[\S4.5, Th\'eor\`eme~1(a)]{Demazure}, this operator is
independent of the chosen reduced decomposition.

For a simple root \(\alpha\), let \(P_\alpha\supset B\) be the corresponding
standard rank-one parabolic.  In equivariant Chow theory,
\cite[Theorem~6.3]{BrionEquivariant} identifies \(\partial_\alpha\) with the
operator \(q_\alpha^*\circ(q_\alpha)_*\) for
\(q_\alpha:\Gamma/B\to\Gamma/P_\alpha\), and proves its compatibility with
equivariant proper push-forwards.  Brion's
\cite[Proposition~6.4 and its proof]{BrionEquivariant} identifies the
iterated operators associated with reduced words with Schubert push-forwards.

We shall use the elementary \(\Ch(BT_\Gamma)^W\)-linearity
\begin{equation}\label{eq:demazure-invariant-linearity}
  \partial_w(uf)=u\partial_w(f)
\end{equation}
for \(u\in\Ch(BT_\Gamma)^W\), \(w\in W\), and \(f\in\Ch(BT_\Gamma)\).

Let \(W_P\subset W\) be the Weyl group of the Levi subgroup of \(P\).
Since \(P\) corresponds to a spin node, \(W_P\) is of type \(A_{n-1}\).
Let \(w^P:=w_0^Ww_0^{W_P}\), the element of maximal length among the minimal
representatives of \(W/W_P\).  Then
\(\ell_W(w^P)=\dim(\Gamma/P)=\binom n2\), where \(\ell_W\) denotes the
Coxeter length function.

Choose a reduced decomposition of \(w^P\) and form the associated
Bott--Samelson variety \(Z_{w^P}\).  Its composite map
\(\beta:Z_{w^P}\to\Gamma/P\) is birational: its image is the Schubert variety
indexed by the maximal coset representative, namely all of \(\Gamma/P\).
Let \(\pi:\Gamma/P\to\Spec F\) be the structural morphism.  For
\(a\in\Ch_P(\Spec F)=\Ch_\Gamma(\Gamma/P)\), the projection formula gives
\(\beta_*\beta^*(a)=a\).  Repeated application of Brion's rank-one
push-forward formula along the Bott--Samelson tower therefore gives
\begin{equation}\label{eq:parabolic-demazure}
  \res\bigl(\pi_*(a)\bigr)=\partial_{w^P}(a_{T_\Gamma}),
  \qquad a\in\Ch_P(\Spec F),
\end{equation}
where \(a_{T_\Gamma}\) is the torus restriction of \(a\).  This is the
parabolic Demazure push-forward formula used below.

\begin{lemma}\label{lem:parabolic-lifting}
The image of the restriction map \(\Ch(BP)\to\Ch(BT_\Gamma)\) is the
\(W_P\)-invariant subring
\[
  \Ch(BT_\Gamma)^{W_P}=\F[z,c_1,\ldots,c_n].
\]
In particular, every \(W\)-invariant torus class has a lift to \(\Ch(BP)\).
\end{lemma}

\begin{proof}
Let \(\mathcal U\) be the rank-\(n\) \(P\)-representation given by the
tautological maximal isotropic subspace.  Its \(T_\Gamma\)-weights are
\(x_1,\ldots,x_n\), so the restrictions of its Chern classes are
\(c_1,\ldots,c_n\).  The \(P\)-character on the highest-weight line \(\ell\)
has first Chern class \(z\).  Thus the image contains
\(\F[z,c_1,\ldots,c_n]\).

Let \(L\) be a Levi subgroup of \(P\) containing \(T_\Gamma\).  Homotopy invariance
for the unipotent radical identifies \(\Ch(BP)\) with \(\Ch(BL)\), and the
restriction image is therefore fixed by the Weyl group \(W_P=N_L(T_\Gamma)/T_\Gamma\).
Since \(W_P\simeq\mathfrak S_n\) permutes the \(x_i\) and fixes \(z\), the
fundamental theorem of symmetric polynomials gives
\(\Ch(BT_\Gamma)^{W_P}=\F[z,c_1,\ldots,c_n]\).  This proves equality.  Finally,
\(\Ch(BT_\Gamma)^W\subset\Ch(BT_\Gamma)^{W_P}\), so every \(W\)-invariant class has the
asserted lift.
\end{proof}

\begin{lemma}\label{lem:cone-demazure}
Let \(\theta\in\Ch_P(\Spec F)\), and write \(\theta_{T_\Gamma}\in\Ch(BT_\Gamma)\) for its
torus restriction.  Then
\[
  \res J(\theta)=\partial_{w^P}(\theta_{T_\Gamma}\Delta),
\]
where \(J\) is the push-forward in \eqref{eq:spinor-cone-pushforward},
\(\res\) is the restriction map in \eqref{eq:torus-restriction}, and
\(\Delta=c^{T_\Gamma}_{2^{n-1}-1}(S^+/\ell)\), equivalently
\(\Delta=\prod_{0\ne I\subset\{1,\ldots,n\},\ |I|\text{ even}}
(z+\sum_{i\in I}x_i)\).
\end{lemma}

\begin{proof}
Consider the factorization
\[
\begin{tikzcd}[column sep=normal,row sep=normal]
\widetilde C=\Gamma\times_P\ell
  \arrow[r,hook,"i"] \arrow[dr,"j"']
& E=\Gamma\times_P S^+
  \arrow[d,"q"] \\
& S^+ .
\end{tikzcd}
\]
Here \(q\) is induced by \([g,s]\mapsto gs\).  Since \(S^+\) is a
\(\Gamma\)-representation, the map \([g,s]\mapsto(gP,gs)\) is a
\(\Gamma\)-equivariant isomorphism \(E\simeq\Gamma/P\times S^+\), under
which \(q\) becomes the second projection.

The map \(i\) is a regular embedding of codimension \(2^{n-1}-1\), with
normal bundle \(\Gamma\times_P(S^+/\ell)\).
On equivariant Chow groups, the factorization and homotopy-invariance
identifications give
\[
\begin{tikzcd}[column sep=large,row sep=large]
\Ch_\Gamma(\widetilde C)
  \arrow[r,"i_*"] \arrow[d,"\simeq"']
& \Ch_\Gamma(E)
  \arrow[r,"q_*"] \arrow[d,"\simeq"]
& \Ch_\Gamma(S^+)
  \arrow[d,"\simeq"] \\
\Ch_P(\Spec F)
  \arrow[r,"{\cdot\, c^P_{2^{n-1}-1}(S^+/\ell)}"']
& \Ch_\Gamma(\Gamma/P)
  \arrow[r,"(\pi)_*"']
& \Ch(B\Gamma),
\end{tikzcd}
\]
where \(\pi\colon\Gamma/P\to\Spec F\) is the structural morphism.  Hence
\[
  J(\theta)
  =(\pi)_*\!\left(c^P_{2^{n-1}-1}(S^+/\ell)\,\theta\right).
\]
Restriction of equivariant Chow groups to \(T_\Gamma\) is compatible with proper
push-forward.  The class \(c^P_{2^{n-1}-1}(S^+/\ell)\,\theta\) restricts
to \(\Delta\theta_{T_\Gamma}\).

By \eqref{eq:parabolic-demazure}, for \(a\in\Ch_P(\Spec F)\) with torus
restriction \(a_{T_\Gamma}\), one has
\(\res\bigl(\pi_*(a)\bigr)=\partial_{w^P}(a_{T_\Gamma})\).
For \(a=c^P_{2^{n-1}-1}(S^+/\ell)\,\theta\), one has
\(a_{T_\Gamma}=\Delta\theta_{T_\Gamma}\); this identity therefore gives
\(\res J(\theta)=\partial_{w^P}(\theta_{T_\Gamma}\Delta)\).
\end{proof}

\begin{proof}[Proof of Theorem~\ref{thm:intro-spinor-cone}]
Let \(u\in\Ch(BT_\Gamma)^W\).  By Lemma~\ref{lem:parabolic-lifting}, choose a lift
\(\widetilde u\in\Ch_P(\Spec F)\) of \(u\).  By
Lemma~\ref{lem:cone-demazure},
\(\res J(\widetilde u)=\partial_{w^P}(u\Delta)\).  Since \(u\) is
\(W\)-invariant, it follows from \eqref{eq:demazure-invariant-linearity} that
\(\partial_{w^P}(u\Delta)=u\partial_{w^P}(\Delta)\).

It remains to identify \(\partial_{w^P}(\Delta)\).  By
Lemma~\ref{lem:cone-demazure} applied to \(1\in\Ch_P(\Spec F)\), together with
Proposition~\ref{prop:normalization}, one has \(\partial_{w^P}(\Delta)=E_n\).
Combining these steps,
\[
  \res J(\widetilde u)=\partial_{w^P}(u\Delta)=u\partial_{w^P}(\Delta)=uE_n.
\]
Thus \(uE_n\) is the restriction of the class
\(J(\widetilde u)\in\Ch(B\Gamma)\).  Since \(u\in\Ch(BT_\Gamma)^W\) was arbitrary,
\(E_n\cdot\Ch(BT_\Gamma)^W\subset
\operatorname{Im}\bigl(\Ch(B\Gamma)\to \Ch(BT_\Gamma)^W\bigr)\).
This proves the theorem.
\end{proof}

\begin{proof}[Proof of Corollary~\ref{cor:rank-five-class}]
For \(n=5\), one has \(\delta_5=2^4-1-\binom52=5\).  By \eqref{eq:spinor-cone-En}, the subsets occurring
in \(E_5\) are precisely the four-element subsets, and therefore
\(E_5=\prod_{k=1}^5(c_1+x_k)\).  Since
\(\prod_{k=1}^5(T+x_k)=T^5+c_1T^4+c_2T^3+c_3T^2+c_4T+c_5\), setting
\(T=c_1\) and reducing modulo two gives
\[
  E_5=c_5+c_1c_4+c_1^2c_3+c_1^3c_2.
\]

Let \(\Gamma=\Gamma^+(10)\), let \(T_\Gamma\subset\Gamma\) be the standard
special-Clifford torus, and let \(J=J_5\).  Let \(P\subset\Gamma\) be the
stabilizer of the highest-weight line, and let \(\mathcal U\) be the
tautological rank-five \(P\)-representation.  Set
\(\widetilde u=c_2(\mathcal U)c_3(\mathcal U)\in\Ch_P^5(\Spec F)\), and let
\(u=c_2c_3\in\Ch(BT_\Gamma)^W\) be its torus restriction.  Since
\(\widetilde u\) is a lift of \(u\), Theorem~\ref{thm:intro-spinor-cone} gives
\[
  \res J(\widetilde u)
  =uE_5=c_2c_3\bigl(c_5+c_1c_4+c_1^2c_3+c_1^3c_2\bigr).
\]
Hence the right-hand side lies in the image of
\(\Ch(B\Gamma)\to\Ch(BT_\Gamma)^W\).

Let \(T=T_\Gamma\cap\Spin(10)\). The semisimple torus \(T\) is cut out inside
\(T_\Gamma\) by the quotient character
\[
  f_0=2z-x_1-\cdots-x_5.
\]
Thus \(f_0|_T=0\). Since \(f_0\) reduces modulo two to
\(c_1=x_1+\cdots+x_5\), the induced map
\(\Ch(BT_\Gamma)\to\Ch(BT)\) sends \(c_1\) to zero and fixes
\(c_2,c_3,c_4,c_5\). Hence the displayed
class restricts to \(c_2c_3c_5\).  By functoriality, the pullback of the
corresponding class of \(\Ch(B\Gamma)\) to \(\Ch(B\!\Spin(10))\) has
torus restriction \(c_2c_3c_5\).  This proves the corollary.
\end{proof}

\section{The characteristic image of \texorpdfstring{$B\!\Spin(10)$}{BSpin(10)}}\label{sec:full-image}
Let $G=\Spin(10)$ and let $T\subset G$ be the standard split maximal torus; in the rank-five construction above, $T=T_\Gamma\cap G$.  We use the presentations
\begin{equation}\label{eq:spin10-torus-rings}
        \CH(BT)=\frac{\Z[z,x_1,\ldots,x_5]}{(2z-x_1-\cdots-x_5)},\qquad \Ch(BT)=\frac{\F[z,x_1,\ldots,x_5]}{(x_1+\cdots+x_5)}.
\end{equation}
Let $c_i$ denote the elementary symmetric polynomial in the variables $x_1,\ldots,x_5$.  Thus $c_1=0$ in $\Ch(BT)$.

The Weyl group is of type $D_5$: it is generated by permutations of the $x_i$ and even sign changes.  By \cite[Proposition~3.2]{KarpenkoSpinModTorsion}, the mod-two invariant ring is $\Ch(BT)^W=\F[t,c_2,c_3,c_4,c_5]$, where $t=\prod_{|I|\text{ even}}\left(z+\sum_{i\in I}x_i\right)$ is the restriction of the top Chern class of a half-spin representation; the product is over subsets $I\subset\{1,\ldots,5\}$.  This agrees with the rank-five orbit-product calculation in Lemma~\ref{lem:special-clifford-invariants} after imposing $c_1=0$.  Here $\deg t=16$.
Let \(U\) be the subring from \eqref{eq:UR-subrings}; explicitly,
\[
        U=\F[c_2,c_4,c_5,c_3^2,c_3c_5]
        \subset \F[c_2,c_3,c_4,c_5].
\]
We shall use the following rules for the operations fixed above; see, for example,
\cite[Theorem~9.3 and Remark~9.5]{Brosnan}:
\begin{equation}\label{eq:steenrod-rules}
        \St(x)=x+x^2,\qquad
        \St(\alpha\beta)=\St(\alpha)\St(\beta).
\end{equation}
The operations $\St^i$ are additive.  Here $x$ is the first Chern class of
a torus character, and $\alpha,\beta$ are arbitrary classes in the mod-two
Chow rings considered below.  We shall use the component form of
multiplicativity, and the square identities which follow from multiplicativity
in characteristic $2$, without further comment.

\begin{lemma}\label{lem:rank-five-ambient}
With the notation above, one has
\begin{equation}\label{eq:ambient-bound}
        \im\varphi_G\subset U[t].
\end{equation}
Moreover, the total Steenrod operation acts on $c_2,c_3,c_4,c_5$ by
\[
\begin{aligned}
\St(c_2)&=c_2+c_3+c_2^2,
&\St(c_3)&=c_3+(c_5+c_2c_3)+c_3^2,\\
\St(c_4)&=c_4+c_5+c_2c_4+(c_3c_4+c_2c_5)+c_4^2,
&\St(c_5)&=c_5+c_2c_5+c_3c_5+c_4c_5+c_5^2.
\end{aligned}
\]
\end{lemma}

\begin{proof}
By \cite[Proposition~2.1]{KarpenkoSpinModTorsion}, the image of
$\Phi_G$ is contained in the subalgebra of $\CH(BT)^W$ generated by
$\Z[x_1^2,\ldots,x_5^2]^{\mathfrak S_5}$, $c_5$, $f_1,f_2,f_3$, and $t$.
Since $\varphi_G$ is induced from $\Phi_G$ by reduction modulo two, it
remains only to compute the reductions of the $f_i$.

The $f_i$ are defined recursively in
\cite[(4.1)--(4.2)]{KarpenkoMerkurjevGrassmannians}.  Write
$f_i=2zh_i+A_i$, where $A_i=\sum_\alpha m_{i,\alpha}$ is independent of
$z$ and its signed monomial summands are counted with multiplicity.
Taking the $z$-independent part in the recursion gives
\[
  A_{i+1}
   =\sum_{\alpha<\beta}m_{i,\alpha}m_{i,\beta}
   =\bigl(A_i^2-\sum_\alpha m_{i,\alpha}^2\bigr)/2,
  \qquad
  \overline{f_i}=\overline{A_i}.
\]

Let $p_j=e_j(x_1^2,\ldots,x_5^2)$.  Comparing coefficients in
$c(u)c(-u)=\prod_{r=1}^5(1-x_r^2u^2)$ gives
\[
\begin{aligned}
p_1&=c_1^2-2c_2, &
p_2&=c_2^2-2c_1c_3+2c_4,\\
p_3&=c_3^2+2c_1c_5-2c_2c_4, &
p_4&=c_4^2-2c_3c_5.
\end{aligned}
\]
At the first two stages, with multiplicities retained, one has
$\sum_\alpha m_{i,\alpha}^2=A_i(x_1^2,\ldots,x_5^2)$ for $i=1,2$.
Starting from $f_0=2z-c_1$, it follows that
\[
  A_1=c_2,\quad A_2=(c_2^2-p_2)/2=c_1c_3-c_4,\quad A_3=\bigl((c_1c_3-c_4)^2-(p_1p_3-p_4)\bigr)/2.
\]
Modulo two, the last identity gives
$A_3\equiv c_4^2+c_3c_5+c_2c_3^2+c_1c_3c_4+c_1^2c_2c_4+c_1^3c_5\pmod 2$.
Since $c_1=0$ in $\Ch(BT)$, this gives
$\overline{f_1}=c_2$, $\overline{f_2}=c_4$, and
$\overline{f_3}=c_4^2+c_3c_5+c_2c_3^2$.

Finally, the reduction of
$\Z[x_1^2,\ldots,x_5^2]^{\mathfrak S_5}$ is contained in
$\F[c_2^2,c_3^2,c_4^2,c_5^2]\subset U$, while $c_5$, $t$, and the three
classes just computed lie in $U[t]$.  Hence
$\im\varphi_G\subset U[t]$.

For the Steenrod formulas, applying \eqref{eq:steenrod-rules} to the Chern
roots gives
\[
  \sum_{j=0}^5\St(c_j)u^j
  =\prod_{r=1}^5\bigl(1+(x_r+x_r^2)u\bigr).
\]
Expanding the right-hand side, rewriting its coefficients in elementary
symmetric functions, and setting $c_1=0$ gives the four displayed formulas.
\end{proof}

Let $M\subset \F[c_2,c_3,c_4,c_5]$ be the smallest graded Steenrod-stable
subring containing
\[
        c_2^2,
        \qquad c_3^2,
        \qquad c_4^2,
        \qquad c_5,
        \qquad c_2c_3c_5.
\]
Let \(R\) be the square subring from \eqref{eq:UR-subrings}.  The next lemmas
describe the additive quotient \(U/M\).

\begin{lemma}\label{lem:M-mod-c5}
Let
\(\pi:\F[c_2,c_3,c_4,c_5]\longrightarrow \F[c_2,c_3,c_4]\)
be the quotient homomorphism defined by $c_5\mapsto 0$.  Then $\pi(M)=R$.
Moreover, $M\subset U$.
\end{lemma}

\begin{proof}
Let $A=\F[c_2,c_3,c_4,c_5]$ and set
\(
        B=R+(c_5)A\subset A.
\)
We first check that $B$ is stable under the Steenrod operations.  Squaring the
formulas for $\St(c_2),\St(c_3),\St(c_4)$ in
Lemma~\ref{lem:rank-five-ambient} shows that
$\St(c_i^2)=\St(c_i)^2\in B$ for $i=2,3,4$.  The same lemma shows that the ideal
$(c_5)A$ is carried into itself by the total operation.  Thus $\St(B)\subset B$;
since $B$ is graded, every component $\St^i$ preserves $B$.

The defining generators $c_2^2,c_3^2,c_4^2,c_5,c_2c_3c_5$ of $M$ all lie in
$B$.  Hence $M\subset B$.  Applying $\pi$ gives $\pi(M)\subset \pi(B)=R$,
while the reverse inclusion $R\subset\pi(M)$ is immediate from the generators
$c_2^2,c_3^2,c_4^2$.  Thus $\pi(M)=R$.

It remains to prove $M\subset U$.  Since $M\subset B$, it is enough to show
$B\subset U$.  We have $R\subset U$.  Every monomial of $(c_5)A$ lies in $U$:
if the exponent of $c_3$ is even, use the generator $c_3^2$; if that exponent is
odd, use one factor $c_3c_5$ and the remaining factors $c_2,c_4,c_3^2,c_5$.
Thus $(c_5)A\subset U$, so $B\subset U$ and therefore $M\subset U$.
\end{proof}

\begin{lemma}\label{lem:U-normal-form}
As an additive $R$-module quotient, every class in $U/M$ has a unique representative of the form
\[
        A c_2+B c_4+C c_2c_4,
        \qquad A,B,C\in R.
\]
\end{lemma}

\begin{proof}
We first show that every $c_5$-divisible monomial in $U$ belongs to $M$.
Since $c_5\in M$, Lemma~\ref{lem:rank-five-ambient} gives
\[
  c_2c_5=\St^2(c_5),
  \qquad
  c_3c_5=\St^3(c_5),
  \qquad
  c_4c_5=\St^4(c_5),
\]
so these three classes lie in $M$.  The class $c_2c_3c_5$ belongs to $M$ by
definition.  Using \eqref{eq:steenrod-rules} and
Lemma~\ref{lem:rank-five-ambient}, one has
$\St^4(c_2c_5)=c_3^2c_5+c_2c_4c_5+c_2^3c_5$ and
$\St^5(c_2c_5)=c_3c_4c_5+c_2c_5^2+c_2^2c_3c_5$.
It follows that $c_2c_4c_5$ and $c_3c_4c_5$ belong to $M$.  Finally,
\[
        \St^4(c_2c_3c_5)=c_2c_3c_4c_5+c_2^3c_3c_5.
\]
Since the second term is in $M$, we get $c_2c_3c_4c_5\in M$.  Thus the eight
classes
\[
        c_5,\ c_2c_5,\ c_4c_5,\ c_2c_4c_5,\ c_3c_5,\ c_2c_3c_5,\ c_3c_4c_5,\ c_2c_3c_4c_5
\]
belong to $M$.  Multiplying them by arbitrary elements of $R$ and by powers of
$c_5$ gives every monomial of $U$ that is divisible by $c_5$.

Modulo such monomials, $U$ becomes $\F[c_2,c_4,c_3^2]$, which is a free
$R$-module with basis $1,c_2,c_4,c_2c_4$.  Since $R\subset M$, the basis vector
$1$ is zero in $U/M$.  If an $R$-linear combination of $c_2,c_4,c_2c_4$ lies in
$M$, then its image under $\pi$ lies in $\pi(M)=R$; freeness over $R$ forces all
three coefficients to vanish.  Together with $M\subset U$, this proves the
asserted existence and uniqueness.
\end{proof}

\begin{lemma}\label{lem:halfspin-low-chern}
Let $S^+$ be a half-spin representation of $\Spin(10)$, and write
$s_i=c_i(S^+)|_T$ and $t=s_{16}$.  Then $s_i=0$ for $1\le i<8$.
Consequently, $\St^i(t^r)=0$ for $1\le i<8$ and $r\ge0$.
\end{lemma}

\begin{proof}
The weights of $S^+$, viewed in $\Ch(BT)$ via
\eqref{eq:spin10-torus-rings}, are $z+\sum_{i\in I}x_i$ with
$I\subset\{1,\ldots,5\}$ of even cardinality.  Let
$h_I=\sum_{i\in I}x_i$.  Modulo the relation $x_1+\cdots+x_5=0$, the elements
$h_I$ form a four-dimensional $\F$-vector subspace $H\subset\Ch^1(BT)$.  Indeed,
the even-parity subspace of $\F^5$ has dimension four, and its image in
$\langle x_1,\ldots,x_5\rangle/(x_1+\cdots+x_5)$ is injective, since
$(1,1,1,1,1)$ has odd parity.  Thus the weights of $S^+$ are precisely $z+h$,
with $h\in H$, and hence
\[
  1+s_1+\cdots+s_{16}=\prod_{h\in H}(1+z+h).
\]
Let $\Psi_H(y)=\prod_{h\in H}(y+h)$.  Since $H$ is an $\F$-vector space,
$\Psi_H(y)$ is additive; hence
\[
        \Psi_H(y)=y^{16}+q_8y^8+q_{12}y^4+q_{14}y^2+q_{15}y,
\]
where $q_j$ is homogeneous of degree $j$ in the variables $x_i$.  To see this,
induct on $\dim H$.  If $H=H'\oplus \F v$, then
\[
        \Psi_H(y)=\Psi_{H'}(y)\Psi_{H'}(y+v)
              =\Psi_{H'}(y)^2+\Psi_{H'}(v)\Psi_{H'}(y),
\]
so only powers of $y$ which are powers of two occur.

Now $1+s_1+\cdots+s_{16}=\Psi_H(1+z)$.  Since
$(1+z)^{2^m}=1+z^{2^m}$ in characteristic two, we get
\[
        \Psi_H(1+z)=1+z^{16}+q_8(1+z^8)+q_{12}(1+z^4)+q_{14}(1+z^2)+q_{15}(1+z).
\]
Every positive-degree term appearing here has total degree at least eight.
Hence $s_i=0$ for $1\le i<8$.

By the splitting principle and \eqref{eq:steenrod-rules}, a rank $N$ vector
bundle $V$ satisfies
\[
        \St(c_N(V))=c_N(V)c(V).
\]
After restriction to $T$, this gives $\St^i(t)=t s_i=0$ for $1\le i<8$.
Using \eqref{eq:steenrod-rules} once more, the same vanishing holds for powers:
$\St^i(t^r)=0$ for $1\le i<8$ and $r\ge0$.
\end{proof}

\begin{proof}[Proof of Theorem~\ref{thm:intro-full-image}]
We first prove the inclusion $M[t]\subset \im\varphi_G$.  The classes
$c_2^2,c_3^2$, and $c_4^2$ lie in the image: they are the reductions modulo two
of the Pontryagin classes, equivalently the restrictions of the even Chern
classes of the standard orthogonal representation.  The class
$c_5=x_1\cdots x_5$ belongs to the image for $\Spin(10)$ by
\cite[Lemma~A.1]{KarpenkoSpinModTorsion}.  The class $c_2c_3c_5$ belongs to the
image by Corollary~\ref{cor:rank-five-class}.  Since the image is stable under
Steenrod operations, $M\subset \im\varphi_G$.  The class $t$ is the torus
restriction of the top Chern class of a half-spin representation, hence also
belongs to the image.  Therefore $M[t]\subset \im\varphi_G$.

It remains to prove the reverse inclusion.  By \eqref{eq:ambient-bound},
$\im\varphi_G\subset U[t]$.  We prove that every element of $U[t]$ whose
$\St^1$ and $\St^3$ again lie in $U[t]$ already belongs to $M[t]$.  This is
enough because the image is Steenrod-stable.

By Lemma~\ref{lem:U-normal-form}, every element $x\in U$ has a unique residue
modulo $M$ of the form
\[
        x\equiv A c_2+B c_4+C c_2c_4\pmod {M},
        \qquad A,B,C\in R.
\]

Now suppose $x\in U$ and $\St^1(x)\in U$.  Since $R$ is generated by squares,
$\St^1$ vanishes on $R$.  Using $\St^1(c_2)=c_3$ and
$\St^1(c_4)=c_5$, we get
\[
        \St^1(x)
        \equiv
        A c_3+B c_5+C(c_3c_4+c_2c_5)
        \pmod {M}.
\]
The terms $Bc_5$ and $Cc_2c_5$ lie in $U$.  The possible terms outside $U$ are
$Ac_3+Cc_3c_4$.  No monomial in $Ac_3$ can cancel a monomial in $Cc_3c_4$,
because the former has even $c_4$-exponent and the latter has odd
$c_4$-exponent.  Since neither type contains a factor $c_5$, both types are
outside $U$.  Thus $\St^1(x)\in U$ forces $A=C=0$.  So every such $x$ is
congruent modulo $M$ to $Bc_4$ with $B\in R$.

Assume in addition that $\St^3(x)\in U$.  Odd Steenrod operations vanish on
$R$, and \eqref{eq:steenrod-rules} gives
\[
        \St^3(Bc_4)=B\St^3(c_4)+\St^2(B)\St^1(c_4)
        =B(c_3c_4+c_2c_5)+\St^2(B)c_5.
\]
The terms $Bc_2c_5$ and $\St^2(B)c_5$ lie in $U$.  The term $Bc_3c_4$ lies
outside $U$ unless $B=0$, again because it has odd $c_3$-exponent and no
factor $c_5$.  Therefore $B=0$, and hence $x\in M$.

We now pass from $U$ to $U[t]$.  Let $y=\sum_{r=0}^N t^r y_r$, with
$y_r\in U$, be an element of the image.  Since the image is Steenrod-stable,
$\St^1(y),\St^3(y)\in U[t]$.  Lemma~\ref{lem:halfspin-low-chern} gives
$\St^i(t^r)=0$ for $1\le i\le3$ and $r\ge0$.
Hence all mixed terms in the Cartan formulas vanish, and
\[
        \St^1(y)=\sum_{r=0}^N t^r\St^1(y_r),
        \qquad
        \St^3(y)=\sum_{r=0}^N t^r\St^3(y_r).
\]
Since $U[t]$ is a polynomial ring in $t$ over $U$, it follows that
$\St^1(y_r),\St^3(y_r)\in U$ for every $r$.  The preceding argument gives
$y_r\in M$ for every $r$.  Hence $y\in M[t]$.  Therefore
$\im\varphi_G\subset M[t]$.  Together with the first inclusion, this proves
the mod-two image formula.

It remains to identify the integral image.  The torsion index of $\Spin(10)$ is
two, and the cokernel of restriction is killed by the torsion index
\cite[Theorems~0.1 and~1.3(1)]{TotaroSpin}.  Hence
$2\CH(BT)^W\subset\im\Phi_G$.
Let $x\in\CH(BT)^W$ satisfy $\rho(x)\in M[t]$.  By the mod-two image formula,
there is a class $\bar a\in\Ch(B\!\Spin(10))$ whose restriction is $\rho(x)$.
Choose an integral lift $\widetilde a\in\CH(B\!\Spin(10))$ of $\bar a$ and set
$d=x-\Phi_G(\widetilde a)$.  Then
$d\in 2\CH(BT)\cap\CH(BT)^W$.
Since $\CH(BT)$ is torsion-free,
\[
  2\CH(BT)\cap\CH(BT)^W=2\CH(BT)^W.
\]
Indeed, if $d=2b$ and $wd=d$ for every $w\in W$, then
$2(wb-b)=0$, so $wb=b$.  Thus $d\in2\CH(BT)^W\subset\im\Phi_G$, and therefore
$x\in\im\Phi_G$.  This proves
$\rho^{-1}(M[t])\subset\im\Phi_G$.

The reverse inclusion follows by reducing modulo two and applying the mod-two
image formula.  Since the kernel of
$\CH(B\!\Spin(10))\to\CH(BT)$ is the torsion ideal, the final identification of
$\CH(B\!\Spin(10))/\mathrm{tors}$ follows.
\end{proof}

\section*{Statements and Declarations}

\subsection*{Funding}
This work was supported by the National Research Foundation of Korea (NRF)
grant funded by the Korean government (MSIT) (No.~RS-2026-25471364).

\subsection*{Data availability}
Data sharing is not applicable to this article as no datasets were generated
or analysed during the current study.

\subsection*{Competing interests}
The author has no relevant financial or non-financial interests to disclose.

\end{document}